\documentclass[12pt]{amsart}
\usepackage[margin=1in]{geometry}
\usepackage[utf8]{inputenc}

\usepackage{verbatim} 
\usepackage{color}
\usepackage{tikz,hyperref}

\usepackage{amsmath}
\usepackage{amsfonts}
\usepackage{amsthm}
\usepackage[backend=bibtex]{biblatex}

\usepackage{xcolor}
\usepackage{tikz}
\usepackage{tikz-cd}
\usetikzlibrary{decorations.pathreplacing,decorations.markings}

\tikzset{
  on each segment/.style={
    decorate,
    decoration={
      show path construction,
      moveto code={},
      lineto code={
        \path [#1]
        (\tikzinputsegmentfirst) -- (\tikzinputsegmentlast);
      },
      curveto code={
        \path [#1] (\tikzinputsegmentfirst)
        .. controls
        (\tikzinputsegmentsupporta) and (\tikzinputsegmentsupportb)
        ..
        (\tikzinputsegmentlast);
      },
      closepath code={
        \path [#1]
        (\tikzinputsegmentfirst) -- (\tikzinputsegmentlast);
      },
    },
  },
  mid arrow/.style={postaction={decorate,decoration={
        markings,
        mark=at position .5 with {\arrow[#1]{stealth}}
      }}},
}

\newtheorem{thm}{Theorem}[section]
\newtheorem{df}[thm]{Definition}
\newtheorem{eg}[thm]{Example}

\newtheorem{prp}[thm]{Proposition}

\newtheorem*{mainthm*}{Theorem \ref{main thm}}

\newcommand{\1}{\mathbb{1}}
\newcommand{\E}{\mathbb{E}}

\renewcommand{\L}{\mathcal{L}}

\newcommand{\wt}{\mathrm{wt}}

\title{Arborescences of Random Covering Graphs}

\author{Muchen Ju}
\address{(M. Ju) Department of Mathematics, Fudan University, 220 Handan Rd, Shanghai, China}
\email{22300180095@m.fudan.edu.cn}
\author{ Junjie Ni}
\address{(J. Ni) Department of Mathematics, Nankai University, 94 Weijin Rd, Tianjin, China}
\email{ni.509@buckeyemail.osu.edu}
\author{Kaixin Wang}
\address{(K. Wang) Department of Mathematics, Duke University, 120 Science Drive, Durham, NC}
\email{kai.wang@duke.edu}
\author{Yihan Xiao}
\address{(Y. Xiao) Department of Mathematics, Shanghai University, 99 Shangda Rd, Shanghai, China}
\email{
jstbfrnd@shu.edu.cn}

\date{\today}

\begin{document}

\maketitle
\begin{abstract}
    A rooted arborescence of a directed graph is a spanning tree directed towards a particular vertex. A recent work of Chepuri et al. \cite{Sylvester} showed that the arborescences of a covering graph of a directed graph $G$ are closely related to the arborescences of $G$. In this paper, we study the weighted sum of arborescences of a random covering graph and give a formula for the expected value, resolving a conjecture of Chepuri et al. \cite{Sylvester}.
\end{abstract}

\section{Introduction}

For a weighted directed graph $\Gamma$, an \emph{arborescence rooted at a vertex} $v$ is a spanning tree whose edges are directed towards $v$. The weight of an arborescence is the product of the edge weights. We define $A_v(\Gamma)$ to be the weighted sum of all arborescences rooted at $v$.

Covering graphs are central objects in topological graph theory motivated by the notion of covering spaces in topology. A weighted digraph $\tilde \Gamma$ is a covering graph of $\Gamma$ defined in the topological sense (see section \ref{section3} for details). The fundamental work of Gross-Tucker \cite{Tucker} gave a characterization of covering graphs by permutation voltage assignments.

In their study of $R$-systems, Galashin and Pylyavskyy \cite{Galashin} observed that the ratio of arborescences of a covering graph and its base graph $A_{\tilde v}(\tilde \Gamma)\over A_v(\Gamma)$ is invariant of the choice of the root $v$. More recently, Chepuri et al. \cite{Sylvester} further investigated the ratio and proved that it is always an integer-coefficient polynomial in the edge weights of the base graph $\Gamma$ conjecturally with positive coefficients. Following up these works, we study the expected value of the ratio $A_{\tilde v}(\tilde \Gamma)\over A_v(\Gamma)$ when $\tilde \Gamma$ is taken to be a random covering graph. Our main theorem is a formula for its expected value, affirming Conjecture 5.6 of \cite{Sylvester}. 

\begin{thm}\label{main thm}
        Let $\Gamma = (V,E,\wt)$ be a weighted digraph, and $k$ be a positive integer. We choose a permutation voltage $\sigma_e$ in the symmetric group $ S_k$ for each edge $e\in E$ independently and uniformly at random. The random voltages allow us to construct our random $k$-fold covering graph $\tilde{\Gamma}$. Then 

        $$ \E\left[\frac{A_{\tilde v}(\tilde \Gamma)}{A_v(\Gamma)}\right] = \frac 1k \prod_{w\in V} \left( \sum_{\alpha \in E_s(w)} \wt(\alpha)\right)^{k-1}$$

        where $E_s(w)$ is the set of edges in $\Gamma$ with source $w$. 
\end{thm}

The paper is organized as follows. In section \ref{section2} we introduce the relevant background in graph theory, including the Matrix-Tree Theorem. In section \ref{section3}, we state the definition of covering graphs and review the theory of arborescence ratio in \cite{Sylvester}. In section \ref{section4} we prove our main result, Theorem \ref{main thm}. 

\section{Background On Graph Theory}\label{section2}

In this paper, we consider directed multigraphs, or \emph{digraphs} for short. Unless otherwise stated, we consider digraphs with weighted edges.
\begin{df}
    Let $G = (V,E,\wt)$ be a weighted digraph with $V = \{v_1,\cdots,v_{|V|}\}$. For a vertex $v\in V$, define $E_s(v)$ to be the set of edges with source $v$ and $E_t(v)$ to be the set of edges with terminal $v$. The \emph{degree matrix} $D$ of $G$ is the $|V|\times |V|$ diagonal matrix with $$D[v_i, v_i] = \sum_{e\in E_s(v_i)} \wt(e)$$
    
    for all $v_i\in V$. The \emph{adjacency matrix} $A$ of $G$ is defined as $$A[v_i,v_j] = \sum_{e=(v_i\to v_j)} \wt(e) $$ for all $v_i,v_j \in V$. The \emph{Laplacian matrix} $L$ of $G$ is defined as $D-A$.
\end{df}

\begin{eg}\label{eg:basic}
    Let $G$ be a weighted digraph with vertex set $\{1,2,3\}$, with only one edge from $i$ to $ j$ (not necessarily distinct) in $V(G)$ with weight $x_{ij}$ for all $i,j \in \{1,2,3\}$. Then $$ D = \begin{bmatrix}
        x_{11} + x_{12} + x_{13} & 0 & 0 \\
        0 & x_{21} + x_{22} + x_{23} & 0 \\
        0 & 0 & x_{31} + x_{32} + x_{33}
    \end{bmatrix}$$

    $$ A = \begin{bmatrix}
        x_{11} & x_{12} & x_{13} \\
        x_{21} & x_{22} & x_{23} \\
        x_{31} & x_{32} & x_{33}
    \end{bmatrix}.$$

    Hence $$ L = D - A = \begin{bmatrix}
         x_{12} + x_{13} & -x_{12} & -x_{13} \\
        -x_{21} & x_{21} + x_{23} & -x_{23} \\
        -x_{31} & -x_{32} & x_{31} + x_{32} 
    \end{bmatrix}.$$
\end{eg}

Note that the Laplacian is always a singular matrix, since the sum of entries in each row is 0.

\begin{df}
    An \emph{arborescence} $T$ of a weighted digraph $\Gamma$ rooted at $v\in V$ is a spanning tree directed towards $v$. The \emph{weight} of an arborescence is the product of the weights of its edges:

    $$ \wt(T) = \prod_{e\in T} \wt(e).$$

    We define $A_v(\Gamma)$ to be the sum of weights of all arborescences.
\end{df}
The Matrix-Tree Theorem \cite{Fomin} relates the minors of the Laplacian and the arborescences. 
\begin{thm}  [Matrix-Tree Theorem \cite{Fomin}] 
    Let $G$ be a weighted digraph with vertex set $\{1,\cdots,n\}$.  Then 

    $$ A_i(G) = \det(L(G)_i^i) $$

    for all $i\in \{1,\cdots,n\}$. Here $L(G)_i^i$ is the Laplacian matrix $L(G)$ with row $i$ and column $i$ removed.
\end{thm}

\begin{eg}
    In example \ref{eg:basic} with $n=3$, 

    $$ L_1^1 = \begin{bmatrix}
        x_{21} + x_{23} & -x_{23}\\
        -x_{32} & x_{31} + x_{32}
    \end{bmatrix}.$$

    Then $$\det(L_1^1) = (x_{21}+x_{23})(x_{31}+x_{32}) - (-x_{23})(-x_{32}) = x_{21}x_{31} + x_{23}x_{31} + x_{21}x_{32} = A_1(G).$$

    The last equality holds since there are 3 arborescences rooted at $1$: namely $2,3\to 1, 2\to 3\to 1$ and $3\to 2\to 1$. 
\end{eg}

\section{Covering Graphs}\label{section3}

A \emph{$k$-fold cover} of a weighted digraph $\Gamma = (V, E,\wt)$ is a weighted digraph $\tilde{\Gamma} = (\tilde{V}, \tilde{E},\wt)$ that is a $k$-fold covering space of $G$ in the topological sense that preserves edge weight. That is, we require that a lifted edge in the covering graph to have the same weight as its corresponding base edge in the base graph. In order to use this definition, we need to find a way to formally topologize directed graphs in a way that encodes edge orientation. To avoid this, we instead give a more concrete alternative definition of a covering graph.

\begin{df}  Let $k$ be a positive integer and $\Gamma = (V,E,\wt)$ be a weighted digraph. A weighted digraph $\tilde{\Gamma} = (\tilde{V}, \tilde{E}, \wt)$ is a \emph{$k$-fold covering graph} of $\Gamma$ if there exists a projection map $\pi: \tilde{\Gamma} \rightarrow \Gamma$ such that
\begin{enumerate}
    \item $\pi$ maps vertices to vertices and edges to edges;
    \item $|\pi^{-1}(v)| = |\pi^{-1}(e)| = k$ for all $v \in V, e \in E$;
    \item If $\tilde e\in \tilde E$ is an edge from $\tilde v$ to $\tilde w$, then $\pi(\tilde e) $ is an edge from $\pi(\tilde v)$ to $\pi(\tilde w)$;
    \item For all $\tilde{e} \in \tilde{E}$, we have $\wt(\tilde{e}) = \wt(\pi(\tilde{e}))$;
    \item $\pi$ is a local homeomorphism.  Equivalently, $|E_s(\tilde{v})| = |E_s(\pi(\tilde{v}))|$ and $|E_t(\tilde{v})| = |E_t(\pi(\tilde{v}))|$ for all $\tilde{v} \in \tilde{V}$.
\end{enumerate}
When we refer to $\tilde{\Gamma}$ as a covering graph of  $\Gamma$, we fix a distinguished projection $\pi: \tilde{\Gamma} \rightarrow \Gamma$. 
\end{df}

 We do not require a covering graph to be connected. However, disconnected graphs contain no arborescences, so our main quantity of interest $A_{\tilde v}(\tilde{\Gamma})$ is always $0$ in the disconnected case. 

\begin{df} Fix an integer $k$. A \emph{weighted permutation-voltage graph} $\Gamma = (V, E, \wt, \nu)$ is a weighted directed multigraph with each edge $e$ also labeled by a permutation $\nu(e) = \sigma_e\in S_k$, the symmetric group on $k$ letters. This labeling is called the \emph{voltage} of the edge $e$. Note that the voltage of an edge $e$ is independent of the weight of $e$.
\end{df}

\begin{df}\label{df:coveringGraph} Let $\Gamma = (V,E,\wt)$ be a weighted digraph with a permutation in $\sigma_e \in S_k$ for every $e\in E$, the \emph{$k$-fold covering graph $\tilde{\Gamma} = (\tilde{V}, \tilde{E}, \wt)$ associated to} $(\sigma_e)_{e\in E}$ is a digraph with vertex set $\tilde{V} = V \times \{1,2,\ldots,k\}$ and edge set
\begin{align*}
    \tilde{E}:=\left\{\left(v \times x, w \times \sigma_e(x)\right) : x \in \{1,\ldots,k\}, e=(v,w)\in E\right\}.
\end{align*}
For an edge $e=(v,w)$ of weight $\wt(e)$, all corresponding $\left(v \times x, w \times \sigma_e(x)\right)$ also have weight $\wt(e)$. 
\end{df}

\begin{thm}[\cite{Tucker}, Gross-Tucker's characterization]
    Every covering graph of $\Gamma$ can be constructed from a permutation voltage graph $\Gamma$. 
\end{thm}

By this characterization result, we can use only definition \ref{df:coveringGraph} to analyze $k$-fold covering graphs from now on.

\begin{eg}\label{ex:perm}
Let $\Gamma$ be the permutation-voltage graph shown in Figure~\ref{fig:perm-voltage-graph}, where edges labeled $(w, \sigma)$ have edge weight $w$ and voltage $\sigma \in S_3$. Then we can construct a $3$-fold cover $\tilde{\Gamma}$, with vertices $(v, y) = v^y$ and with edges labeled by weight, is shown in Figure~\ref{fig:perm-derived-graph}.
\begin{figure}[htp]\label{fig:derived-graph}
    \centering
    \begin{tikzpicture}
   [scale=0.65,line width=1.2]
    \coordinate (1) at (0, 3);
    \coordinate (2) at (3/1.71, 0);
    \coordinate (3) at (-3/1.71, 0);
    
    \coordinate (4) at (0, 3 + 0.4);
    \draw (1) arc(270:360+270:0.4);

    \path [draw = red, postaction = {on each segment = {mid arrow = red}}]
    (1)--(2);
     \path [draw = blue, postaction = {on each segment = {mid arrow = red}}]
     (3)--(1);
     \path [draw = green, postaction = {on each segment = {mid arrow = red}}]
    (2) to [bend left] (3);
    
     \path [draw = purple, postaction = {on each segment = {mid arrow = red}}]
    (3) to [bend left] (2);
    
    \draw[fill] (1) circle [radius=0.1];
    \node at (0.5, 3) {$1$};
    \draw[fill] (2) circle [radius=0.1];
    \node at (3/1.71 + 0.5, 0) {2};
    \draw[fill] (3) circle [radius=0.1];
    \node at (-3/1.71-0.5, 0) {3};

    \node at (1 + 0.4, 3+ 0.5) {$(a, 321)$};
    \node at (3/1.71/2 + 1.1, 3/2) {$(b, 231)$};
    \node at (-3/1.71/2 - 1.2, 3/2) {$(d, 123)$};
    \node at (0, 0.85 + 0.1) {$(e, 132)$};
    \node at (0, -0.85 - 0.1) {$(c, 123)$};
    \end{tikzpicture}
    \caption{A permutation-voltage graph $\Gamma$.}
    \label{fig:perm-voltage-graph}
\end{figure}
\begin{figure}[htp]
    \centering
  \begin{tikzpicture}[scale = 0.7, line width=1.2]
    \coordinate (1) at (0, 3+4);
    
    \coordinate (2) at (3/1.71, 4);
    \coordinate (3) at (-3/1.71, 4);
    
    \coordinate (4) at (-4, 3);
    
    \coordinate (5) at (3/1.71 - 4, 0);
    \coordinate (6) at (-3/1.71 - 4, 0);
    
    \coordinate (7) at (4, 3);
    
    \coordinate (8) at (3/1.71 + 4, 0);
    \coordinate (9) at (-3/1.71 + 4, 0);
    
\draw[fill] (1) circle [radius=0.1] ;
    \draw[fill] (2) circle [radius=0.1];
    \draw[fill] (3) circle [radius=0.1];
    \draw[fill] (4) circle [radius=0.1];
    \draw[fill] (5) circle [radius=0.1];
    \draw[fill] (6) circle [radius=0.1];
    \draw[fill] (7) circle [radius=0.1];
    \draw[fill] (8) circle [radius=0.1];
    \draw[fill] (9) circle [radius=0.1];
    
    \node at (0, 7.5) {$1^1$};
    
    \node at (3/1.71, 3.6) {$2^1$};
    \node at (-3/1.71 -0.4, 4) {$3^1$};
    
    \node at (-4 + 0.5, 3+0.3) {$1^2$};
    
    \node at (3/1.71 - 4, -0.5) {$2^2$};
    \node at (-3/1.71 - 4 -0.5, 0) {$3^2$};
    
    \node at (4 + 0.5, 3) {$1^3$};
    
    \node at (3/1.71 + 4 + 0.5, 0) {$2^3$};
    \node at (-3/1.71 + 4, -0.5) {$3^3$};
    
    \path[draw = black, postaction = {on each segment = {mid arrow = red}}]
    
    (1) to (7)
    (7) to [bend right] (1);
    
    \draw (4) arc(-45:360-45:0.4);
    
    \path[draw = green, postaction = {on each segment = {mid arrow = red}}]
    (5) -- (6)
    (8) to (9)
    (2) to [bend left] (3);
    
    \path[draw = red, postaction = {on each segment = {mid arrow = red}}]
    (1)  -- (5)
    (4) -- (8)
    (7) -- (2);
    
    \path[draw = blue, postaction = {on each segment = {mid arrow = red}}]
    (3) -- (1)
    (6) -- (4)
    (9) -- (7);
    
    \path[draw = purple, postaction = {on each segment = {mid arrow = red}}]
    (3) to [bend left] (2)
    (6) to [bend right] (8)
    (9) -- (5);
    
    \node at (1.5, 5.1) {$a$};
    \node at (2.8, 5.8) {$a$};
    \node at (-4.2, 4) {$a$};
    
    \node at (0.6, 1.9) {$b$};
    \node at (-1, 2.9) {$b$};
    \node at (2.7, 3.2) {$b$};
    
    \node at (0, 4-0.3) {$c$};
    \node at (-4, 0.3) {$c$};\node at (4, 0.3) {$c$};
    
    \node at (-3/1.71/2 - 0.5, 3/2+4) {$d$};
    \node at (-3/1.71/2 - 0.5-4, 3/2) {$d$};
    \node at (-3/1.71/2 + 0.4+4, 3/2) {$d$};
    
    \node at (0, 4+0.3) {$e$};
    \node at (0, -1.3) {$e$};
    \node at (0, -0.3) {$e$};
    
    \draw[fill] (1) circle [radius=0.1] ;
    \draw[fill] (2) circle [radius=0.1];
    \draw[fill] (3) circle [radius=0.1];
    \draw[fill] (4) circle [radius=0.1];
    \draw[fill] (5) circle [radius=0.1];
    \draw[fill] (6) circle [radius=0.1];
    \draw[fill] (7) circle [radius=0.1];
    \draw[fill] (8) circle [radius=0.1];
    \draw[fill] (9) circle [radius=0.1];
    \end{tikzpicture}
    \caption{The derived covering graph $\tilde{\Gamma}$ of $\Gamma$ in Figure~\ref{fig:perm-voltage-graph}. Edge colors denote correspondence to the edges of $\Gamma$ via the projection map.}
\label{fig:perm-derived-graph}
\end{figure}
\end{eg}

\newpage Given $v\in V(\Gamma)$, for any   $\tilde v \in V(\tilde \Gamma)$ that projects to $v$, $\frac{A_{\tilde v}(\tilde \Gamma)}{A_v(\Gamma)}$ is independent of $\tilde v$. We call this ratio $R(\tilde \Gamma)$. This ratio has an explicit formula in \cite{Sylvester} as follows.

\begin{df}\label{lowerright}
Let $\{v_1,\cdots,v_n\}$ be the set of vertices of our digraph $\Gamma$, let $\tilde \Gamma$ be a $k$-fold cover of $\Gamma$, where vertex $v_i$ is lifted to $v_i^1,\cdots,v_i^k$.
Define $n(k-1)\times n(k-1)$ matrices $D$ and $\mathcal{A}$ with basis vectors $v_1^2,\cdots,v_1^{k}, v_2^2,\cdots,v_2^{k}, \cdots, v_n^2, \cdots, v_n^{k}$,  as follows.
\[\mathcal{A}[v_i^t,v_j^r]=\sum_{e=(v_i^t\to v_j^r)}\wt(e)-\sum_{e=(v_i^t\to v_j^1)}\wt(e)\]
\[D[v_i^t,v_i^t]=\sum_{e\in E_s(v_i^t)}\wt(e)\]
for $1<t,r\leq k$ and $1\le i,j\le n$. Finally, we define
\[ \mathcal{L}(\tilde{\Gamma}):=D-\mathcal{A}\]
\end{df}

 to be the \emph{voltage Laplacian} of $\tilde{\Gamma}$.

\begin{thm}[\cite{Sylvester}]\label{bigboi}
	Let $\Gamma = (V, E, \wt)$ be an weighted digraph, and let $\tilde{\Gamma}$ be a $k$-fold cover of $\Gamma$ such that each lifted edge is assigned the same weight as its base edge. Denote by $\mathcal{L}(\Gamma)$ the voltage Laplacian of $\Gamma$. Then for any vertex $v$ of $\Gamma$ and any lift $\tilde{v}$ of $v$ in  $\tilde \Gamma$ , we have
	\begin{align*}
		\frac{A_{\tilde{v}}(\tilde{\Gamma})}{A_v(\Gamma)} = \frac{1}{k}\det  [\mathcal{L}(\tilde{\Gamma})].
	\end{align*}
\end{thm}

\begin{eg}
Returning to figures \ref{fig:perm-voltage-graph} and \ref{fig:perm-derived-graph} as an example. The voltage Laplacian can be written as a linear transformation on the basis vectors $v_1^2,v_1^3,v_2^2,v_2^3,v_3^2,v_3^3$, in this order, with matrix
 $$  \mathcal{L} = \begin{bmatrix}
        b & 0 & 0 & -b & 0 & 0 \\
        a & 2a+b & b & b & 0 & 0\\
        0 & 0 & c & 0 & -c & 0\\
        0 & 0 & 0 & c & 0 & -c \\
        -d & 0 & 0 & -e & d+e & 0\\
        0 & -d & -e & 0 & 0 & d+e\\
    \end{bmatrix}.$$
\end{eg}

We now restate our main theorem.
\begin{mainthm*}
Let $\Gamma = (V,E,\wt)$ be a weighted digraph, and $k$ be a positive integer. We choose a permutation voltage $\sigma_e$ in the symmetric group $ S_k$ for each edge $e\in E$ independently and uniformly at random. The random voltages allow us to construct our random $k$-fold covering graph $\tilde{\Gamma}$. Then 

        $$ \E\left[\frac{A_{\tilde v}(\tilde \Gamma)}{A_v(\Gamma)}\right] = \frac 1k \prod_{w\in V} \left( \sum_{\alpha \in E_s(w)} \wt(\alpha)\right)^{k-1}$$

        where $E_s(w)$ is the set of edges in $\Gamma$ with source $w$. 
\end{mainthm*}

\begin{eg}

Consider the following graph of 2 vertices:
\[
\Gamma = 
\begin{tikzpicture}[baseline=-0.65ex,scale=1.2]
    \node[fill=black, circle, inner sep=1pt, label=below:1] (1) at (0,0) {};
    \node[fill=black, circle, inner sep=1pt, label=below:2] (2) at (2,0) {};
    
    \draw[->, thick, green!70!black] (0,0.4) arc[start angle=90, end angle=450, radius=0.2];
    \node at (-0.3,0.5) {\small $a$};

    \draw[->, thick, red] (2,0) -- (0,0);
    \node at (1,0.2) {\small $b$};
\end{tikzpicture}
\]

There are four 2-fold covering graphs of $\Gamma$:

\begin{center}
\begin{tikzpicture}[scale=1.1, every node/.style={scale=1}]
\begin{scope}
\node[fill=black, circle, inner sep=1pt, label=left:$1^2$] (1a) at (0,2) {};
\node[fill=black, circle, inner sep=1pt, label=left:$1^1$] (1b) at (0,0) {};
\node[fill=black, circle, inner sep=1pt, label=right:$2^2$] (2a) at (1.5,2) {};
\node[fill=black, circle, inner sep=1pt, label=right:$2^1$] (2b) at (1.5,0) {};

\draw[->, thick, red] (2a) -- (1a);
\draw[->, thick, red] (2b) -- (1b);
\draw[->, thick, green!70!black, mid arrow] (1a) to[bend left=40] (1b);
\draw[->, thick, green!70!black, mid arrow]  (1b) to[bend left=40] (1a);

\node at (0.75,-0.8) {$\mathcal{A}_{1^1}(\widetilde{\Gamma}) = ab^2$};
\end{scope}
\begin{scope}[xshift=3.5cm]
\node[fill=black, circle, inner sep=1pt, label=left:$1^2$] (1a) at (0,2) {};
\node[fill=black, circle, inner sep=1pt, label=left:$1^1$] (1b) at (0,0) {};
\node[fill=black, circle, inner sep=1pt, label=right:$2^2$] (2a) at (1.5,2) {};
\node[fill=black, circle, inner sep=1pt, label=right:$2^1$] (2b) at (1.5,0) {};

\draw[->, thick, red] (2a) -- (1b);
\draw[->, thick, red] (2b) -- (1a);
\draw[->, thick, green!70!black, mid arrow] (1a) to[bend left=30] (1b);
\draw[->, thick, green!70!black, mid arrow]  (1b) to[bend left=30] (1a);

\node at (0.75,-0.8) {$\mathcal{A}_{1^1}(\widetilde{\Gamma}) = ab^2$};
\end{scope}

\begin{scope}[xshift=7cm]
\node[fill=black, circle, inner sep=1pt, label=left:$1^2$] (1a) at (0,2) {};
\node[fill=black, circle, inner sep=1pt, label=left:$1^1$] (1b) at (0,0) {};
\node[fill=black, circle, inner sep=1pt, label=right:$2^2$] (2a) at (1.5,2) {};
\node[fill=black, circle, inner sep=1pt, label=right:$2^1$] (2b) at (1.5,0) {};

\draw[->, thick, red] (2a) -- (1a);
\draw[->, thick, red] (2b) -- (1b);
\draw[->, thick, green!70!black] (1a) to[out=45,in=135,looseness=8] (1a);
\draw[->, thick, green!70!black] (1b) to[out=225,in=315,looseness=8] (1b); 

\node at (0.75,-0.8) {$\mathcal{A}_{1^1}(\widetilde{\Gamma}) = 0$};
\end{scope}

\begin{scope}[xshift=10.5cm]
\node[fill=black, circle, inner sep=1pt, label=left:$1^2$] (1a) at (0,2) {};
\node[fill=black, circle, inner sep=1pt, label=left:$1^1$] (1b) at (0,0) {};
\node[fill=black, circle, inner sep=1pt, label=right:$2^2$] (2a) at (1.5,2) {};
\node[fill=black, circle, inner sep=1pt, label=right:$2^1$] (2b) at (1.5,0) {};

\draw[->, thick, red] (2a) -- (1b);
\draw[->, thick, red] (2b) -- (1a);
\draw[->, thick, green!70!black] (1a) to[out=225,in=315,looseness=8] (1a);
\draw[thick, green!70!black] (1b) to[out=225,in=315,looseness=8] (1b);

\node at (0.75,-0.8) {$\mathcal{A}_{1^1}(\widetilde{\Gamma}) = 0$};
\end{scope}
\end{tikzpicture}
\end{center}
Therefore, the expected value of $A_{1^1}(\tilde{\Gamma})$ is $\frac 14 (ab^2+ab^2+0+0) = \frac 12 ab^2$. Since $A_1(\Gamma)$ is deterministically equal to $b$, we can see that $\E[\frac{A_{1^1}(\tilde{\Gamma})}{A_1(\Gamma)}] = \frac 12 ab$. Since $\sum_{\alpha \in E_s(1)} \wt(\alpha)=a$ and $\sum_{\alpha \in E_s(2)} \wt(\alpha)=b$, this is consistent with theorem \ref{main thm}.
\end{eg}

\section{Proof of Theorem 1.1} \label{section4}









Recall that Theorem \ref{main thm} gives a formula for the expected value of the arborescences ratio $A_{\tilde v}(\tilde \Gamma)\over A_v(\Gamma)$ as follows
 $$ \E[R(\tilde{\Gamma})]=\E\left[\frac{A_{\tilde v}(\tilde \Gamma)}{A_v(\Gamma)}\right] = \frac 1k \prod_{w\in V} \left( \sum_{\alpha \in E_s(w)} \wt(\alpha)\right)^{k-1}$$

        where $E_s(w)$ is the set of edges in $\Gamma$ with source $w$. 

Say the vertices of $\Gamma$ are $v_1,\cdots,v_n$ and edges $\{e = (v_i \to v_j, \sigma_e)\},$ where $\sigma_e\in S_k$ is a random permutation.

Recall $$ R(\tilde{\Gamma}):= \frac{A_{\tilde v}(\tilde \Gamma)}{A_v(\Gamma)}= \frac 1k \det \mathcal{L} = \frac 1k \det (D-\mathcal{A})$$ 

where $D,\mathcal{A}$ are square matrices with rows/columns labeled $v_1^2,\cdots,v_1^{k}, v_2^2,\cdots,v_2^{k}, $
\\ $\cdots, v_n^2, \cdots, v_n^{k}$, in this order.

Here $D$ is a deterministic diagonal matrix satisfying $$D[v_i^t, v_i^t] = \sum_{e = (v_i\to v_j)} \wt(e)$$

and $\mathcal{A}$ is a random matrix with

$$ \mathcal{A}[v_i^t, v_j^r] = \sum_{e = (v_i\to v_j), \sigma_e(t) = r} \wt(e) - \sum_{e = (v_i\to v_j), \sigma_e(t) = 1} \wt(e) $$

where the $\sigma_e \in S_k$ are chosen independently and uniformly at random.

Thus, Theorem \ref{main thm} is equivalent to showing that our random matrix $\mathcal{A}$ satisfies $$\E [\det(D-\mathcal{A})] = \det(D).$$

Define $S := \{v_1^2,\cdots,v_1^{k}, v_2^2,\cdots,v_2^{k}, \cdots, v_n^2, \cdots, v_n^{k}\} $.

For each edge $e$ of $\Gamma$, let $X_{e,a,b} = \1_{\{ \sigma_e(a)=b\}}$, then we can write 

$$\mathcal{A}[v_i^t, v_j^r] = \sum_{e = (v_i\to v_j)} \wt(e) (X_{e,t,r} - X_{e,t,1}) $$

where we group together $X_{e,t,r},X_{e,t,1}$; we naturally define $Y_{e,t,r} := X_{e,t,r} - X_{e,t,1}$.

\begin{eg}\label{eg:section4}
    Let $\Gamma$ be a digraph on one vertex, with two loops weighted $a,b$. For $k=3$, the matrix $\mathcal{L}=D-\mathcal{A}$ is as follows:

$$\begin{bmatrix}
    a+b - a(X_{a,2,2}-X_{a,2,1}) - b(X_{b,2,2}-X_{b,2,1})   & - a(X_{a,2,3}-X_{a,2,1}) - b(X_{b,2,3}-X_{b,2,1})   \\
    - a(X_{a,3,2}-X_{a,3,1}) - b(X_{b,3,2}-X_{b,3,1})   & a+b - a(X_{a,3,3}-X_{a,3,1}) - b(X_{b,3,3}-X_{b,3,1}) 
\end{bmatrix}$$

$$=\begin{bmatrix}
    a+b - aY_{a,2,2}- bY_{b,2,2}   & - aY_{a,2,3} - bY_{b,2,3}   \\
    - aY_{a,3,2} - bY_{b,3,2}  & a+b - aY_{a,3,3} - bY_{b,3,3}
\end{bmatrix}.$$

Similarly, the matrix $\L(\tilde\Gamma)$ of the digraph in example \ref{ex:perm} is of the form:

 $$   \L(\tilde\Gamma) =\begin{bmatrix}
         a+b - aY_{a,2,2} & - aY_{a,2,3} & -bY_{b,2,2} & -bY_{b,2,3} & 0 & 0 \\
        - aY_{a,3,2} & a+b-aY_{a,3,3} & -bY_{b,3,2} & -bY_{b,3,3} & 0 & 0\\
        0 & 0 & c & 0 & -cY_{c,2,2} & -cY_{c,2,3}\\
        0 & 0 & 0 & c & -cY_{c,3,2} & -cY_{c,3,3} \\
        -dY_{d,2,2} & -dY_{d,2,3} & -eY_{e,2,2} & -eY_{e,2,3} & d+e & 0\\
        -dY_{d,3,2} & -dY_{d,3,3} & -eY_{e,3,2} & -eY_{e,3,3}& 0 & d+e\\
    \end{bmatrix}.$$

\end{eg}

  To compute $\det(\L)$, we first choose a permutation $\rho$ on $S= \{v_i^t \mid 1\le i\le n, 2\le t\le k\}$ to expand; the product corresponding to $\rho$ is $$sgn(\rho)\prod\limits_{\substack{1\le i\le n \\ 2\le t\le k}} \mathcal{L}(\tilde{\Gamma})_{v_i^t, \rho(v_i^t)} .$$ 
  
   We define a \emph{term} to be either $\wt(e)$ that can be found on the diagonal of $\mathcal{L}(\tilde{\Gamma})$ or a $-\wt(e)Y_{e,*,*} $. We expand each $\mathcal{L}(\tilde{\Gamma})_{v_i^t, \rho(v_i^t)}$ into terms; for example, in example \ref{ex:perm}, $\mathcal{L}(\tilde{\Gamma})_{1,1} $ has three terms: $a,b$, and $-aY_{a,2,2}$.

We define a \emph{monomial} to be a product of terms that arise when we expand $\prod\limits_{\substack{1\le i\le n \\ 2\le t\le k}} \mathcal{L}(\tilde{\Gamma})_{v_i^t, \rho(v_i^t)}$. For example, in example \ref{ex:perm}, we can have $\rho = 135624 = v_1^2v_2^2v_3^2v_3^3v_1^3v_2^3$ , and we choose terms $a,-bY_{b,3,2}, -cY_{c,2,2}, -cY_{c,3,3}, -dY_{d,2,3}, -eY_{e,3,3}$ from $\mathcal{L}_{1,1}, \mathcal{L}_{2,3}, \mathcal{L}_{3,5}, \mathcal{L}_{4,6}, \mathcal{L}_{5,2}, \mathcal{L}_{6,4}$, respectively. The corresponding monomial is, then, $-abc^2de Y_{b,3,2}Y_{c,2,2}Y_{c,3,3}Y_{d,2,3}Y_{e,3,3}$.

For a monomial, we define $F$ to be the set of $wt(e)$'s from the diagonal in the monomial. For each edge $e$, let $n_e$ be the number of terms of the form $-\wt(e) Y_{e,*,*}$ in the monomial, namely $-\wt(e) Y_{e,i_{e,1} , j_{e,1}},-\wt(e) Y_{e,i_{e,2},j_{e,2}}, \cdots, -\wt(e)Y_{e,i_{e,n_e},j_{e,n_e}} $ where $i_{e,\alpha} \in \{2,\cdots,k\}, j_{e,\alpha} \in \{1,\cdots,k\}$.  Thus, the monomial can be written as $$\left(\prod_{e' \in F} \wt(e')\right) \cdot  \prod_{e\in E}  \left(\prod_{l=1}^{n_e}(-\wt(e))Y_{e,i_{e,l},j_{e,l}} \right)  .$$ 

Note that in the formula, both $i_{e,1},\cdots,i_{e,n_e}$ and $j_{e,1},\cdots,j_{e,n_e}$ are pairwise distinct, and take values in $\{2,\cdots,n\}$. 

Observe $\det(\mathcal{L})$ is just a signed sum all of these monomials, where the sign comes from the sign of the permutation $\rho$.

Fix a sequence of nonnegative integers $(n_e)_{e\in E}$. We consider monomials where we choose exactly $n_e$ terms of the form $-\wt(e)Y_{e,*,*}$ for every $e\in E$. Every such monomial can be written as 

$$ \left(\prod_{e' \in F} \wt(e')\right) \cdot  \prod_{e\in E} (-\wt(e))^{n_e} \cdot \left(\prod_{l=1}^{n_e} Y_{e,i_{e,l},j_{e,l}} \right). $$

This means when we sum all these monomials, every term shares the same $\left(\prod_{e' \in F} \wt(e')\right) \cdot  \prod_{e\in E} (-\wt(e))^{n_e} $ factor, so we only need to look at the expected value of the (signed) sum of the $\prod_{e\in E} \prod_{l=1}^{n_e} Y_{e,i_{e,l},j_{e,l}} $'s.

If $n_e = 0$ for all $e\in E$, then we only choose terms of the form $\wt(e)$ in the monomial; we did not choose any term of the form $-\wt(e)Y_{e,*,*}$. In this case, we have chosen terms arising only from $D$, and we must have $\rho=id\in S_k$. Thus the sum of all these terms is $\det(D)$. We did not need to take an expectation, since these terms are deterministic.

It remains to show if there exists $e\in E$ such that $n_e>0$, then the corresponding sum is $0$.

\begin{prp}\label{prp:expect}
    Let $e=(x\to y)\in E$ be a fixed edge where $x,y\in [n]$, $t \in \{1,\cdots,k-1\}$ be a fixed integer. Fix distinct $i_{1},\cdots,i_{t} \in \{2,\cdots,k\}$, distinct $j_1,\cdots,j_{t} \in \{2,\cdots,k\}$ and a bijection $\pi \colon \{i_1,\cdots,i_t\} \to \{j_1,\cdots,j_t\}$. Then $$ \E\left[\prod_{l=1}^t Y_{e,i_l,j_l}\right]$$ depends only on $t$; furthermore when $t=1$ the this is equal to 0.
\end{prp}

\begin{proof}
    Note that we only need to choose $\sigma_e \in S_k$ uniformly at random here. Observe if $\tau \colon \{i_1,\cdots,i_t\} \to [k]$ is a map, then $$ \E\left[\prod_{l=1}^t X_{e,i_l, \tau(i_l)}\right] = \begin{cases}
    0 & \text{ if } \tau \text{ is not injective; } \\
    \frac{(k-t)!}{k! } & \text{ if } \tau \text{ is injective }
\end{cases}$$

Hence 
\begin{align*}
    & \E\left[\prod_{l=1}^t Y_{e,i_l, \pi(i_l)}\right] \\
    &= \E\left[\prod_{l=1}^t (X_{e,i_l, \pi(i_l)} - X_{e,i_l,1}) \right]\\ &= \E\left[ \prod_{l=1}^t X_{e,i_l, \pi(i_l)}\right] - \sum_{z=1}^t \E\left[\prod_{\substack{1\le l\le t \\ l\ne z}} X_{e,i_l,\pi(i_l)} \cdot X_{e,i_z,1}\right]  \\
    &= (1-t)\frac{(k-t)!}{k!}
\end{align*}

\end{proof}

For this choice of $(n_e)_{e\in E}$, choose an edge $e_0$ such that $n_{e_0}>0$.

Case 1: The monomials have exactly one term of the form $Y_{e_0,*,*} $ , i.e. $n_{e_0}=1$. In this case, since the $\sigma_e$'s are independent random permutations, the expected value of every monomial factor through a term $\E[Y_{e_0,*,*} ]$, which is zero, so the contribution of this monomial to the expected value is 0.

Case 2: The monomials have $\ge 2$ terms of the form $Y_{e_0,*,*} $, i.e. $n_{e_0}\ge2$. 

Fix $i_1<\cdots<i_{n_{e_0}}\in \{2,\cdots,n\}$. Among the monomials associated to our fixed $(n_e)_{e\in E}$, we focus on the monomials that contain a product $\prod_{l=1}^{n_{e_0}} Y_{e_0,i_{e_0,l},j_{e_0,l}}$ for some distinct integers $j_{e_0,1},\cdots,j_{e_0,n_{e_0}} \in \{2,\cdots,n\}$. We will show the expected value of the signed sum of these monomials is zero. 

Again, by the independence of $\sigma_e$'s, we have $$ \E\left[ \prod_{e\in E} \prod_{l=1}^{n_e} Y_{e,i_{e,l},j_{e,l}}\right] = \prod_{e\in E} \E\left[ \prod_{l=1}^{n_e} Y_{e,i_{e,l},j_{e,l}}\right] = \E\left[ \prod_{l=1}^{n_{e_0}} Y_{e_0,i_{e_0,l},j_{e_0,l}}\right] \cdot \prod_{e\in E\setminus \{e_0\}}\E\left[ \prod_{l=1}^{n_e} Y_{e,i_{e,l},j_{e,l}}\right]$$

By Proposition \ref{prp:expect}, 

$$\E\left[Y_{e_0,i_{e_0,1},j_{e_0,1}}Y_{e_0,i_{e_0,2},j_{e_0,2}}\prod_{l=3}^{n_{e_0}}Y_{e_0,i_{e_0,l},j_{e_0,l}}\right] = \E\left[ Y_{e_0,i_{e_0,1},j_{e_0,2}}Y_{e_0,i_{e_0,2},j_{e_0,1}}\prod_{l=3}^{n_{e_0}}Y_{e_0,i_{e_0,l},j_{e_0,l}} \right] $$ 

Thus, \begin{align*} 
 &\E\left[ Y_{e_0,i_{e_0,1},j_{e_0,1}} Y_{e_0,i_{e_0,2},j_{e_0,2}} \prod_{l=3}^{n_{e_0}} Y_{e_0,i_{e_0,l},j_{e_0,l}}\right] \cdot \prod_{e\in E\setminus \{e_0\}}\E\left[ \prod_{l=1}^{n_e} Y_{e,i_{e,l},j_{e,l}}\right] \\= &\E\left[ Y_{e_0,i_{e_0,1},j_{e_0,2}} Y_{e_0,i_{e_0,2},j_{e_0,1}} \prod_{l=3}^{n_{e_0}} Y_{e_0,i_{e_0,l},j_{e_0,l}}\right] \cdot \prod_{e\in E\setminus \{e_0\}}\E\left[ \prod_{l=1}^{n_e} Y_{e,i_{e,l},j_{e,l}}\right]
\end{align*}
  Denote $e=(x\to y)$ where $x,y\in [n]$. The term $ \prod_{e\in E} \prod_{l=1}^{n_e} Y_{e,i_{e,l},j_{e,l}}$ appears in the signed sum from some permutation $\rho \in S_{n(k-1)}$. Furthermore, such $\rho$ must satisfy $\rho(v_x^{i_{e,l}})=v_y^{j_{e,l}}$ for all $e\in E, l=1,\cdots,n_{e}$, but $\rho$ is not necessarily unique. In the computation of the signed sum, we need to account this term for every such $\rho$.
  

  Write $e_0 = (x_0\to y_0)$. Consider any particular choice of $\rho$. Observe that the term $$Y_{e_0,i_{e_0,1},j_{e_0,2}} Y_{e_0,i_{e_0,2},j_{e_0,1}} \prod_{l=3}^{n_{e_0}} Y_{e_0,i_{e_0,l},j_{e_0,l}} \allowbreak\prod_{e\in E\setminus \{e_0\}} \prod_{l=1}^{n_e} Y_{e,i_{e_0,l},j_{e_0,l}} $$ is also in the sum; it comes from the permutation $(v_{y_0}^{j_{e_0,1}},v_{y_0}^{j_{e_0,2}})\circ\rho$. We define a map $f$ that sends the pair $$\left(\rho,Y_{e_0,i_{e_0,1},j_{e_0,1}} Y_{e_0,i_{e_0,2},j_{e_0,2}} \prod_{l=3}^{n_{e_0}} Y_{e_0,i_{e_0,l},j_{e_0,l}} \prod_{e\in E\setminus \{e_0\}} \prod_{l=1}^{n_e} Y_{e,i_{e,l},j_{e,l}} \right) $$ to the pair $$ \left((v_{y_0}^{j_{e_0,1}},v_{y_0}^{j_{e_0,2}})\circ\rho,Y_{e_0,i_{e_0,1},j_{e_0,2}} Y_{e_0,i_{e_0,2},j_{e_0,1}} \prod_{l=3}^{n_{e_0}} Y_{e_0,i_{e_0,l},j_{e_0,l}} \prod_{e\in E\setminus \{e_0\}} \prod_{l=1}^{n_e} Y_{e,i_{e,l},j_{e,l}} \right).$$ Since $f$ is an involution, it defines a pairing (of pairs). And since the permutations in each pair have opposite signs, the signed sum of expected value of each pair is $0$. The conclusion follows.

\begin{eg}

In example \ref{eg:section4}, we arrived at 
    $$\mathcal{L} = \begin{bmatrix}
    a+b - aY_{a,2,2} - bY_{b,2,2}   & - aY_{a,2,3} - bY_{b,2,3}   \\
    - aY_{a,3,2} - bY_{b,3,2}   & a+b - aY_{a,3,3} - bY_{b,3,3} 
\end{bmatrix}$$

The monomial $(-aY_{a,*,*})(-bY_{b,*,*})$ is associated to the sequence $n_a=n_b=1$. The monomial $(-aY_{a,*,*})(-aY_{a,*,*})$ is associated to the sequence $n_a=2, n_b=0$. The monomial $ (-aY_{a,*,*})b$ is associated to the sequence $n_a=1, n_b=0$. 

The sequence  $n_a=n_b=1$ is associated with four monomials: $$(-aY_{a,2,2})(-bY_{b,3,3}),  (-aY_{a,2,3})(-bY_{b,3,2}), (-aY_{a,3,2})(-bY_{b,2,3}), (-aY_{a,3,3})(-bY_{b,2,2})$$ The contribution of each monomial to the expected value is $0$ since $\E[Y_{a,i,j}]=0$.

The sequence  $n_a=2,n_b=0$ is associated with monomials $(-aY_{a,2,2})(-aY_{a,3,3})$, $(-aY_{a,2,3})(-aY_{a,3,2})$. We pair the permutation $(id, Y_{a,2,2}Y_{a,3,3})$ with $((2,3),Y_{a,2,3}Y_{a,3,2})$. The signed sum of these is $0$. This corresponds to expanding 

$$ \E[ (-aY_{a,2,2})(-aY_{a,3,3}) - (-aY_{a,2,3})(-aY_{a,3,2})] = 0$$

as part of the computation of the expected value of the determinant of the matrix.
\end{eg}

\section*{Acknowledgements}
We are grateful to Sylvester W. Zhang for introducing the problem and providing valuable guidance on the research. We thank Tao Gui for helpful conversations. We thank PACE 2024 for providing an unforgettable opportunity for us to conduct mathematical researches. This program is partially supported by NSFC grant 12426507.

\printbibliography

@article{Sylvester,
  author    = {Sunita Chepuri and CJ Dowd and Andrew Hardt and Gregory Michel and Sylvester W. Zhang and Valerie Zhang},
  title     = {Arborescences of covering graphs},
  journal   = {Algebraic Combinatorics},
  volume    = {5},
  number    = {2},
  pages     = {319--346},
  year      = {2022},
  doi       = {10.5802/alco.212}
}

@book{Fomin,
  author    = {Sergei Fomin and Richard Stanley},
  title     = {Enumerative Combinatorics Volume 2},
  publisher = {Cambridge University Press},
  year      = {1999}
}

@article{Galashin,
  author    = {Pavel Galashin and Pavlo Pylyavskyy},
  title     = {$R$-systems},
  journal   = {Selecta Mathematica (New Series)},
  volume    = {25},
  number    = {2},
  year      = {2019}
}

@article{Tucker,
  author    = {Jonathan Gross and Thomas Tucker},
  title     = {Generating all graph coverings by permutation voltage assignments},
  journal   = {Discrete Mathematics},
  pages     = {273--283},
  month     = aug,
  year      = {1975}
}

\end{document}